\documentclass[12pt, reqno]{amsart}
\usepackage{amsmath}
\usepackage{amssymb}
\usepackage{amsthm}
\usepackage{bm}

\usepackage{epsbox}
\usepackage[dvips]{graphicx}
\usepackage{color}

\numberwithin{equation}{section}
\newtheorem{theorem}{Theorem}[section]
\newtheorem{lemma}[theorem]{Lemma}
\newtheorem{proposition}[theorem]{Proposition}

\newtheorem{example}[theorem]{Example}
\newtheorem{definition}[theorem]{Definition}

\newtheorem{remark}[theorem]{Remark}

\newcommand{\ep}{\varepsilon}

\newcommand{\tht}{\theta}

\newcommand{\ld}{\lambda}

\newcommand{\q}{\quad}

\newcommand{\wh}{\widehat}

\newcommand{\R}{\mathbb{R}}
\newcommand{\N}{\mathbb{N}}
\newcommand{\Z}{\mathbb{Z}}
\newcommand{\Prime}{\mathbb{P}}

\newcommand{\rd}{\mathbb R ^d}

\newcommand{\pn}{\par\noindent}

\newcommand{\mZ}{\mathcal{Z}}

\newcommand{\va}{{\Vec{a}}}

\newcommand{\vc}{{\Vec{c}}}
\newcommand{\vs}{{\Vec{s}}}
\newcommand{\vt}{{\Vec{t}}}
\newcommand{\vk}{{\Vec{k}}}
\newcommand{\vz}{{\Vec{0}}}
\newcommand{\vsig}{{\Vec{\sigma}}}

\allowdisplaybreaks
\begin{document}
\title[Shintani zeta distributions]{Multidimensional Shintani zeta functions and zeta distributions on $\rd$}
\author[T.~ Aoyama and T.~Nakamura]{Takahiro Aoyama* \& Takashi Nakamura}
\address{Department of Mathematics Faculty of Science and Technology \\Tokyo University of Science Noda, CHIBA 278-8510 JAPAN}
\email{{\rm aoyama$\_$takahiro$@$ma.noda.tus.ac.jp, nakamura$\_$takashi$@$ma.noda.tus.ac.jp}}
\maketitle

\begin{abstract}
The class of Riemann zeta distribution is one of the classical classes of probability distributions on $\R$.
Multidimensional Shintani zeta function is introduced and its definable probability distributions on $\rd$ are studied.
This class contains some fundamental probability distributions such as binomial and Poisson distributions.
The relation with multidimensional polynomial Euler product, which induces multidimensional infinitely divisible distributions on $\rd$, is also studied.
\end{abstract}

%%%%%%%%%%%%%%%%%%%%%%%%%%%%%%%%%%%%%%%%%%%%
\section{Introduction}
%%%%%%%%%%%%%%%%%%%%%%%%%%%%%%%%%%%%%%%%%%%%
%%%%%%%%%%%%%%%%%%%%%%%%%%%%%
\subsection{Probability distributions}
%%%%%%%%%%%%%%%%%%%%%%%%%%%%%

There exist many classes of probability distributions such as normal and exponential distributions.
The Fourier transforms of probability distributions are usually called characteristic functions in probability theory.
In the study of probabilistic limit theorems and stochastic processes, analytic methods often appear by treating them.

\vskip3mm

Let $\mu$ be a probability distribution on $\rd$ and its characteristic function $\wh{\mu}(t):=\int_{\rd}e^{{\rm i}\langle t,x\rangle}$ $\mu (dx),\, t\in\rd,$
where $\langle\cdot ,\cdot\rangle$ is the inner product.
In the following, we give some examples of characteristic functions.

\begin{example}\label{ex:CH}
\pn
$(i)$ $($Delta measure.$)$
Let $\mu_{\va} (=\delta_{\va})$ be a delta measure at $\va\in\rd$, then
\begin{equation}\label{eq:delta}
\wh\mu_{\va}(\vt)=\exp{{\rm{i}}\langle\va, \vt\rangle},\q \vt\in\rd.
\end{equation}
\pn
$(ii)$ $($Binomial distribution.$)$ 
Let $d=1$ and $\mu_{Bn}$ be a binomial distribution with parameter $K\in\N$.
Then,
\begin{equation}\label{eq:bi}
\mu_{Bn}(t)=(pe^{{\rm{i}}t}+q)^K,\q t\in\R,
\end{equation}
where $p,q>0$ with $p+q=1$.
\pn
$(iii)$ $($Compound Poisson distribution.$)$
Let $\mu_{{\rm CPo}}$ be a compound Poisson distribution.
Then there exist some $c>0$ and $\rho$, a distribution on $\rd$ with $\rho (\{0\})=0$, such that
\begin{equation}\label{CPcf}
\wh\mu_{{\rm CPo}}(\vt)=\exp \left(c\left(\wh\rho (\vt) -1\right)\right), \q \vt\in\rd. 
\end{equation}
The Poisson distribution is a special case when $d=1$ and $\rho =\delta_1$.
\pn
\end{example}

There is another class of distribution which is defined as follows.

\begin{definition}[Infinitely divisible distribution]
A probability measure $\mu$ on $\rd$ is infinitely divisible if,
for any positive integer $n$, there is a probability measure $\mu_n$ on 
$\rd$
such that
\begin{equation*}
\mu=\mu_n^{n*},
\end{equation*}
where $\mu_n^{n*}$ is the $n$-fold convolution of $\mu_n$.
\end{definition}

This class is known as one of the most important class of distributions in probability theory.
Infinitely divisible distributions are the marginal distributions of stochastic processes having independent and stationary increments 
such as Brownian motion and Poisson processes. 
In 1930's, such stochastic processes were well-studied by P. L\'evy and now we usually call them L\'evy processes.
We can find the detail of L\'evy processes in \cite{S99}.

%%%%%%%%%%%%%%%%%%%%%%%%%%%%%%%%%%%%%%%%%%%%%%%%
\subsection{Riemann and Hurwitz zeta functions}
%%%%%%%%%%%%%%%%%%%%%%%%%%%%%%%%%%%%%%%%%%%%%%%%

Zeta functions play one of the key roles in number theory. 
The Riemann zeta function is regarded as the prototype. 
First results about this function were obtained by L.~Euler in the eighteenth century. 
It is named after B.~Riemann, who in the memoir \lq\lq On the Number of Primes Less Than a Given Magnitude", 
published in 1859, established a relation between its zeros and the distribution of prime numbers.
The definition of the Riemann zeta function is as follows.

\begin{definition}[Riemann zeta function (see, e.g.\,\cite{Apo})]
The Riemann zeta function is a function of a complex variable $s = \sigma + {\rm i}t$, for $\sigma >1$ given by
\begin{align}
\zeta (s) :=& \sum_{n=1}^{\infty} \frac{1}{n^s}\label{eq:ds}\\
 =& \prod_p \Bigl( 1 - \frac{1}{p^s} \Bigr)^{-1} ,\label{eq:eupro}
\end{align}
where the letter $p$ is a prime number,  and the product of $\prod_p$ is taken over all primes.
\end{definition}
It is well-known that the right-hand side of \eqref{eq:ds} is called the Dirichlet series and \eqref{eq:eupro} the Euler product.
The Dirichlet series and the Euler product of $\zeta (s)$ converges absolutely in the half-plane $\sigma >1$ and uniformly in each compact subset of this half-plane. 

By partial summation, we have
$$
\zeta (s) = \sum_{n \le N} \frac{1}{n^s} + \frac{N^{1-s}}{s-1} + s \int_N^{\infty} \frac{[x] -x}{x^{s+1}} dx,
$$
where the sequel $[x]$ denotes the maximal integer less than or equal to $x$. 
The above formula gives the analytic continuation for $\zeta (s)$ to the half-plane $\sigma > 0$ with a simple pole at $s=1$ with residue $1$. 

Next we introduce Dirichlet characters and the Dirichlet $L$-functions. 
Let $q$ be a positive integer. 
A Dirichlet character $\chi$ mod $q$ is a non-vanishing group homomorphism from the group $({\mathbb{Z}}/q{\mathbb{Z}})^*$ of prime residue classes modulo $q$ to ${\mathbb{C}}$. 
The character which is identically one is denoted by $\chi_0$ and is called the principal. 
By setting $\chi (n) = \chi (a)$ for $n \equiv a \mod q$, we can extend the character to a completely multiplicative arithmetic function on ${\mathbb{Z}}$. 
\begin{definition}[Dirichlet $L$-function (see, e.g.\,\cite{Apo})]
For $\sigma >1$, the Dirichlet $L$-function $L(s,\chi)$ attached to a character $\chi$ mod $q$ is given by
\begin{equation}
L (s,\chi) := \sum_{n=1}^{\infty} \frac{\chi (n)}{n^s} = \prod_p \Bigl( 1 - \frac{\chi(p)}{p^s} \Bigr)^{-1} .
\end{equation}
\end{definition}
The Riemann zeta function $\zeta (s)$ may be regarded as the Dirichlet $L$-function to the principal character $\chi_0 \mod 1$. 
It is possible that for values of $n$ coprime with $q$ the character $\chi (n)$ may have a period less than $q$. 
If so, we say that $\chi$ is imprimitive, and otherwise primitive. 
Every non-principal imprimitive character is induced by a primitive character. 
Two characters are non-equivalent if they are not induced by the same character. 
Characters to a common modulus are pairwise non-equivalent. 

It is well-known that if $\chi$ is a non-principal Dirichlet character, the Dirichlet series of $L(s , \chi )$ converges for $\sigma >0$ according to Abel's partial summation. 
We can show that $L(s,\chi)$ is continued analytically to ${\mathbb{C}}$, similarly as the case of the Riemann zeta function, and regular at $s=1$ if and only if $\chi$ is non-principal by partial summation. 
Furthermore, Dirichlet $L$-functions to primitive characters satisfy a functional equation of the Riemann-type. \\

As one of a generalization of $\zeta (s)$, the following function is also well-known.
\begin{definition}[Hurwitz zeta function (see, e.g.\,\cite{Apo})]\label{def:ler}
For $0 < u \le 1$ and $\sigma >1$, the Hurwitz zeta function $\zeta(s,u)$ is defined by
\begin{equation}
\zeta (s,u) := \sum_{n=0}^{\infty} \frac{1}{(n+u)^s}.
\end{equation}
\end{definition}
Note that we obviously have $\zeta (s) = \zeta (s,1)$. 
The function $\zeta (s,u)$ is analytically continuable to the whole complex plane as a meromorphic function with a simple pole at $s=1$. 

%%%%%%%%%%%%%%%%%%%%%%%%%%%%%%%%%%%%%%%%%%%%%%%%
\subsection{Riemann and Hurwitz zeta distributions}
%%%%%%%%%%%%%%%%%%%%%%%%%%%%%%%%%%%%%%%%%%%%%%%%

In probability theory, there exists a class of distribution on $\R$ which is generated by the Riemann zeta function.
First it appears in \cite{Khi} and we can also find it in \cite{GK68}.
What we have known about this distribution is not many.
Few properties are noted in \cite{GK68} and further ones are studied in \cite{Lin}.
Recently, a class of distribution on $\R$ generated by the Hurwitz zeta function is introduced and studied in \cite{Hu06}.
In this section, we mention the Riemann and Hurwitz zeta distributions with some known properties.

Put
\begin{equation*}
f_{\sigma}(t):=\frac{\zeta (\sigma +{\rm i}t)}{\zeta (\sigma)}, \q t\in\R,
\end{equation*}
then $f_{\sigma}(t)$ is known to be a characteristic function. 
(See, e.g.\,\cite{GK68}.)

\begin{definition}[Riemann zeta distribution on $\R$]
A distribution $\mu_{\sigma}$ on $\R$ is said to be a Riemann zeta distribution with parameter $\sigma$ if it has $f_{\sigma}(t)$ as its characteristic function.
\end{definition}

The Riemann zeta distribution is known to be infinitely divisible.
Its L\'evy measure is given of the form as in the following.

\begin{proposition}[See, e.g.\,\cite{GK68}]\label{pro:RD}
Let $\mu_{\sigma}$ be a Riemann zeta distribution on $\R$ with characteristic function $f_{\sigma}(t)$.
Then, $\mu_{\sigma}$ is compound Poisson on $\R$ and
\begin{align*}
\log f_{\sigma}(t)&=
\sum_{p}\sum_{r=1}^{\infty}\frac{p^{-r\sigma}}{r}\left(e^{-{\rm i}rt\log p}-1\right)\\
&=\int_0^{\infty}\left(e^{-{\rm i}tx}-1\right)N_{\sigma}(dx),
\end{align*}
where $N_{\sigma}$ is given by
\begin{align*}
N_{\sigma}(dx)=\sum_{p}\sum_{r=1}^{\infty}\frac{p^{-r\sigma}}{r}\delta_{r\log p}(dx),
\end{align*}
where $\delta_x$ is the delta measure at $x$.
\end{proposition}

Next we mention the Hurwitz zeta distribution.
Put the corresponding normalized function and a discrete one-sided random variable $X_{\sigma,u}$ as follows:
\begin{equation*}
f_{\sigma,u}(t):=\frac{\zeta (\sigma +{\rm i}t,u)}{\zeta (\sigma,u)}, \q t\in\R,
\end{equation*}
and
\begin{equation*}
{\rm Pr}\left(X_{\sigma,u}=\log (n+u)\right)=\frac{(n+u)^{-\sigma}}{\zeta (\sigma,u)}\q{\rm for} \,\,n\in\N\cup\{0\}.
\end{equation*}

Then $f_{\sigma,u}$ is known to be a characteristic function of $-X_{\sigma,u}$.

\begin{proposition}[{\cite[Theorem 1]{Hu06}}]
$(i)$ The Laplace-Stieltjes transform of $X_{\sigma,u}$ is $\Psi_{\sigma,u}(s)=\zeta(\sigma +s,u)/\zeta(\sigma,u)$, $s>1-\sigma$.\\
$(ii)$ The characteristic function of $-X_{\sigma,u}$ is $f_{\sigma,u}$.
\end{proposition}

Therefore, we can define the following distribution.

\begin{definition}[Hurwitz zeta distribution on $\R$]
A distribution $\mu_{\sigma,u}$ on $\R$ is said to be a Hurwitz zeta distribution with parameter $(\sigma,u)$ 
if it has $f_{\sigma,u}$ as its characteristic function.
\end{definition}

The infinite divisibility of $\mu_{\sigma,u}$ is studied in \cite{Hu06}.

\begin{proposition}[{\cite[Theorem 3]{Hu06}}]\label{proHu1}
The Hurwitz zeta distribution $\mu_{\sigma,u}$ is infinitely divisible if and only if
\begin{equation*}
u=\frac{1}{2}\q \mbox{or} \q u=1.
\end{equation*}
\end{proposition}

The L\'evy measure of $\mu_{\sigma,\frac{1}{2}}$ is also given as follows.

\begin{proposition}[{\cite[Theorem 2]{Hu06}}]
The Hurwitz zeta distribution $\mu_{\sigma,\frac{1}{2}}$ is compound Poisson (infinitely divisible) 
with its L\'evy measure $N_{\sigma,\frac{1}{2}}$ given by
\begin{equation*}
N_{\sigma,\frac{1}{2}}(dx)=\sum_{p>2}\sum_{r=1}^{\infty}\frac{p^{-r\sigma}}{r}\delta_{r\log p}(dx),
\end{equation*}
where the first sum is taken over all odd primes $p$.
\end{proposition}

\begin{remark}
We have to note that both the Riemann and Hurwitz zeta distributions are defined in the region of absolute convergence.
The parameter $\sigma$ is always larger than 1 not in the whole complex plane.
\end{remark}

%%%%%%%%%%%%%%%%%%%%%%%%%%%%%%%%%%%%%%%%%%%%%%%%
\subsection{Shintani zeta function}
%%%%%%%%%%%%%%%%%%%%%%%%%%%%%%%%%%%%%%%%%%%%%%%%
As a multiple sum version of Hurwitz zeta function, Barnes \cite{Ba} considered a multiple sum of the form
\begin{equation*}
\zeta_r (s, u \,|\, \Lambda) := \sum_{n_1 ,\ldots, n_r =0}^{\infty} (\lambda_1n_1+\cdots+\lambda_rn_r+u)^{-s},
\q \Re (s) >r \ge 2,
\end{equation*}
where $u$, $\lambda_1, \ldots , \lambda_r$ are complex numbers satisfying some conditions. 
Nowadays this function is called the Barnes $r$-tuple zeta function. 
Barnes proved that the function $\zeta_r (s , u \,|\, \Lambda)$ can be continued meromorphically to the whole $s$-plane and is holomorphic except for simple poles at $s= 1, \ldots, r$. 
Barnes defined the multiple gamma function by $\zeta_r (s, u \,|\, \Lambda)$ and studied its properties. 
Afterwards many mathematicians have studied properties of the Barnes multiple zeta functions (see, for example \cite[Section 1]{Masmz}). 

In order to study the Barnes multiple gamma function, the following generalized Barnes multiple zeta function is introduced.

\begin{equation}
\zeta_S (\vs) := \sum_{n_1 ,\ldots, n_r =0}^{\infty} 
\prod_{l=1}^m \bigl(\lambda_{l1}(n_1+u_1) 
+\cdots+\lambda_{lr}(n_r+u_r) \bigr)^{-s_l}.
\label{eq:Bardef}
\end{equation}

Note that the original motivation of Shintani's research lies in the problem of constructing class fields over algebraic number fields. 
Cassou-Nogu\`es (see, for example \cite{Ca}) who was inspired by Shintani's work, considered those multiple series of the form that the numerator of (\ref{eq:Bardef}) is multiplied by certain roots of unity with $s_1 = \cdots =s_m$. 
She proved its meromorphic continuation and gave applications to $L$-functions and $p$-adic $L$-functions of totally real number fields (see, a survey \cite[Section 2]{Masmz}). 
Imai \cite{Imai} and Hida \cite{Hida} considered this series and the series with more generalized characters in the numerator. 
In \cite[Lemma 2.4.1]{Hida}, it was showed that these multiple series converges absolutely and uniformly on any compact subset in the region $\Re (s_l) > r/m$ for all $1 \le l \le m$. 
Moreover, in \cite[Theorem 2.4.1]{Hida}, it was proved that they can be continued to the whole space ${\mathbb{C}}^m$ as a meromorphic function.

%%%%%%%%%%%%%%%%%%%%%%%%%%%%%%%%%%%%%%%%%%%%%%%%
\subsection{Aim}
%%%%%%%%%%%%%%%%%%%%%%%%%%%%%%%%%%%%%%%%%%%%%%%%

It is well-known that infinitely divisible characteristic functions do not have zeros.
In zeta cases, this property can give us information of zeros of zeta functions which is one of the most important subject in number theory.
Historically, there exist many probability distributions on $\rd$ and multiple zeta functions but we do not see useful zeta distributions on $\rd$.
The purpose of our recent work is to establish zeta distributions on $\rd$ as other well-known distributions and show properties of them including the relationship with number theory.
As a first generalization, in view of the series representations, we have introduced multidimensional Shintani zeta functions and corresponding zeta distributions on $\rd$ in \cite{AN11k}.
In \cite{AN12e}, in view of the Euler products, we also have introduced multidimensional polynomial Euler products as to define infinitely divisible 
zeta distributions on $\rd$, and necessary and sufficient conditions for some of those products to generate compound Poisson characteristic functions are given.

In this paper, adjusting to general number theory further, we reconstruct our previous story of \cite{AN11k} and give new results which could not be obtained in Section 2.
Some important examples of distributions and functions related to infinite divisibility, number theory and zeros of zeta functions are also given and considered.
In Section 3, the relation with multidimensional polynomial Euler products is shown and some important examples of zeta functions related to these new classes are studied as new results.

%%%%%%%%%%%%%%%%%%%%%%%%%%%%%%%%%%%%%%%%%%%
\section{Multidimensional Shintani zeta functions and zeta distributions on $\rd$}
%%%%%%%%%%%%%%%%%%%%%%%%%%%%%%%%%%%%%%%%%%%

%%%%%%%%%%%%%%%%%%%%%%%%%%%%%%%%%%%%%%%%%%%
\subsection{Multidimensional Shintani zeta function}
%%%%%%%%%%%%%%%%%%%%%%%%%%%%%%%%%%%%%%%%%%%

As we have mentioned in Section 1.3, the known zeta distributions are considered only the case on $\R$.
For a generalization of them to $\rd$-valued, we define a new multiple Shintani $L$-function. 

\begin{definition}[Multidimensional Shintani zeta function, $Z_S(\vs)$]\label{Def}
Let $d,m,r\in\N$, $\vs\in\mathbb{C}^d$ and $(n_1, \ldots , n_r)\in\mathbb{Z}_{\ge 0}^{r}$.
For $\ld_{lj}, u_j > 0$, $\vc_l \in {\mathbb{R}}^d$, where $1\le j\le r$ and $1\le l\le m$, and 
a function $\tht (n_1, \ldots , n_r)\in{\mathbb{C}}$
satisfying $|\tht (n_1, \ldots , n_r)| = O((n_1+ \cdots +n_r)^{\ep})$, for any $\ep >0$, 
we define a multidimensional Shintani zeta function given by
\begin{equation}
Z_S (\vs) := \sum_{n_1 ,\ldots, n_r =0}^{\infty} 
\frac{\tht (n_1,\ldots , n_r)}{\prod_{l=1}^m (\lambda_{l1}(n_1+u_1) 
+\cdots+\lambda_{lr}(n_r+u_r) )^{\langle \vc_l,\vs \rangle}}.
\label{eq:def2}
\end{equation}
\end{definition}

This is a multidimensional case of the Shintani multiple zeta functions, 
when the coefficient $\tht (n_1,\ldots , n_r)$ in \eqref{eq:def2} is a product of Dirichlet characters, 
considered by Hida \cite{Hida}. 

\begin{remark}
A similar definition is already introduced in \cite{AN11k}. 
It was defined for some fixed $\ep >0$ not for any $\ep >0$.
By following the general number theory, we renew our definition.
\end{remark}

Put 
$$
\vs:=\vsig +{\rm i}\vt, \q \vsig, \vt\in\rd.
$$
The absolute convergence of $Z_S (\vs)$ is also given as follows.

\begin{theorem}\label{th:sc}
The series defined by $\eqref{eq:def2}$ converges absolutely in the region $\min_{1\le l\le m}$ $\Re\langle \vc_l, \vs\rangle >r/m$. 
\end{theorem}

\begin{proof}
Note that $r\in\N$ and $\sum_{l=1}^m\langle \vc_l,\vsig\rangle \ge r$.
Put $\lambda := \min \{ \lambda_{lj} \}>0$ and $u := \min \{u_j\}>0$. 
Obviously, we have 
$$
(\lambda_{l1}(n_1+u_1) +\cdots+\lambda_{lr}(n_r+u_r) )^{-1} \le
\lambda^{-1} (n_1+\cdots +n_r+ru )^{-1}. 
$$
Therefore, for any $0<\ep <\sum_{l=1}^m\langle \vc_l,\vsig\rangle -r$, it holds that
\begin{equation*}
\begin{split}
&\sum_{n_1 ,\ldots, n_r =0}^{\infty} \left|
\frac{\tht (n_1,\ldots , n_r)}{\prod_{l=1}^m (\lambda_{l1}(n_1+u_1) 
+\cdots+\lambda_{lr}(n_r+u_r) )^{\langle \vc_l,\vs \rangle}}\right|\\
\le & \sum_{n_1 ,\ldots, n_r =0}^{\infty}
\frac{\lambda^{-\sum_{l=1}^m\langle \vc_l,\vsig\rangle}(n_1 + \cdots + n_r +ru)^{\varepsilon}}{\prod_{l=1}^m
(n_1+\cdots+n_r+ru)^{\langle \vc_l,\vsig \rangle}}\\
= &\sum_{n_1 ,\ldots, n_r =0}^{\infty}
\frac{\lambda^{-\sum_{l=1}^m\langle \vc_l,\vsig\rangle}}{(n_1+\cdots+n_r+ru)^{-\varepsilon+\sum_{l=1}^m\langle \vc_l,\vsig\rangle}} \\
\le & \lambda^{-\sum_{l=1}^m\langle \vc_l,\vsig\rangle}\left((ru)^{\varepsilon-\sum_{l=1}^m\langle \vc_l,\vsig \rangle} +
\int_0^{\infty} \!\!\!\!...\! \int_0^{\infty} \frac{dx_1 \cdots
dx_r}{(x_1+\cdots+x_r+ru)^{-\varepsilon+\sum_{l=1}^m\langle \vc_l,\vsig \rangle}}\right)\\
\le & \lambda^{-\sum_{l=1}^m\langle \vc_l,\vsig\rangle}\left((ru)^{\varepsilon-\sum_{l=1}^m\langle \vc_l,\vsig \rangle} 
+C(ru)^{\ep -\sum_{l=1}^m\langle \vc_l,\vsig\rangle +r}\right)<\infty,
\end{split}
\end{equation*}
where
\begin{equation*}
C:=\left(\left(\sum_{l=1}^m\langle \vc_l,\vsig\rangle)-\ep -1\right)
\cdots\left(\sum_{l=1}^m\langle \vc_l,\vsig\rangle -\ep -r\right)\right)^{-1}>0.
\end{equation*}
Thus $Z_S (\vs)$ converges absolutely in the region $\min_{1\le l\le m}\Re \langle \vc_l,\vs \rangle > r/m$.
\end{proof}

Note that it is proved that $Z_S (\vs)$ can be continued to the whole space ${\mathbb{C}}^d$ as a meromorphic function when $d=m$, $c_1 = (1,0,\ldots,0),\dots ,c_m = (0,\ldots,0,1)$ and $\tht (n_1,\ldots , n_r)$ is a product of Dirichlet characters in \cite{Hida}. 

\vskip3mm

Let $\mZ_S$ be the set of all multidimensional Shintani zeta functions $Z_S$ and denote by ${\rm D}_Z\subset \mathbb{C}^d$ the region of absolute convergence of $Z_S\in\mZ_S$ given in Theorem \ref{th:sc}. 
The following is a new result which can not be obtained by our previous definition.

\begin{theorem}\label{th:de}
Let $\vk=(k_1,\dots ,k_d)\in\Z_{\ge 0}^d$ and $Z_S\in\mZ_S$. 
Then,
$$
Z_S^{(\vk)} \in {\mathcal{Z}}_S,
$$
where
$$
Z_S^{(\vk)} (\vs) := \frac{\partial^{k_1}}{\partial^{k_1} s_1} \cdots \frac{\partial^{k_d}}{\partial^{k_d} s_d} Z_S (\vs), \q \vs\in{\rm D}_Z.
$$
\end{theorem}

\begin{proof}
The case when $\vk =\vz$ is trivial.
Let $\vs\in\mathbb{C}^d$ and put $\vc_l := (c_{l1}, \ldots, c_{ld})$.
Define a vector-valued function $T_h(n_1 ,\ldots, $ $n_r)$, where $(n_1,\ldots , n_r)\in\mathbb{Z}_{\ge 0}^{r}$ and $1 \le h \le d$, as follows:
$$
T_h(n_1 ,\ldots, n_r) := \frac{\partial}{\partial s_h} \tht (n_1,\ldots , n_r) \prod_{q=1}^m 
\left(\lambda_{q1}(n_1+u_1)+\cdots+\lambda_{qr}(n_r+u_r)\right)^{-\langle \vc_q,\vs \rangle}.
$$
Then we have, for $(n_1, \ldots , n_r)\in\mathbb{Z}_{\ge 0}^{r}$ and $1 \le h \le d$,
\begin{equation*}
\begin{split}
T_h(n_1 ,\ldots, n_r) =  \frac{\tht (n_1,\ldots , n_r)\sum_{q=1}^m (-c_{qh}) \log (\lambda_{q1}(n_1+u_1)+\cdots+\lambda_{qr}(n_r+u_r))}
{\prod_{l=1}^m(\lambda_{l1}(n_1+u_1) +\cdots+\lambda_{lr}(n_r+u_r) )^{\langle \vc_l,\vs \rangle}}. 
\end{split}
\end{equation*}
Put
\begin{equation}
\theta_h'(n_1,\ldots , n_r):=\sum_{q=1}^m (-c_{qh}) \log \bigl(\lambda_{q1}(n_1+u_1)+\cdots+\lambda_{qr}(n_r+u_r) \bigr).
\label{eq:th'}
\end{equation}
Obviously, we have, for any $\ep >0$,
\begin{equation*}
\begin{split}
\left|\theta_h'(n_1,\ldots , n_r)\right|&\le\sum_{q=1}^m \left| c_{qh} \log \bigl(\lambda_{q1}(n_1+u_1)+\cdots+\lambda_{qr}(n_r+u_r) \bigr) \right|\\
&\le \sum_{q=1}^m \bigl(\lambda_{q1}(n_1+u_1)+\cdots+\lambda_{qr}(n_r+u_r)\bigr) ^{\varepsilon}\\
&\le \left( \sum_{q=1}^m \bigl(\lambda_{q1}(n_1+u_1)+\cdots+\lambda_{qr}(n_r+u_r)\bigr) \right)^{m\varepsilon}
\end{split}
\end{equation*}
for sufficiently large $n_1, \ldots, n_r$. 
Thus $|\theta_h'(n_1,\ldots , n_r)|=O((n_1+ \cdots +n_r)^{m\ep})$.

By Theorem \ref{th:sc}, $Z_S(\vs)$ converges absolutely in ${\rm D}_Z$.
Therefore one has, for $\vs\in {\rm D}_Z$ and $1 \le h \le d$,
\begin{equation*}
\begin{split}
\frac{\partial}{\partial s_h} Z_S (\vs) &= \sum_{n_1 ,\ldots, n_r =0}^{\infty} T_h(n_1 ,\ldots, n_r)\\
&= \sum_{n_1 ,\ldots, n_r =0}^{\infty} \frac{\tht (n_1,\ldots , n_r)\theta_h'(n_1,\ldots , n_r)}
{\prod_{l=1}^m (\lambda_{l1}(n_1+u_1) +\cdots+\lambda_{lr}(n_r+u_r) )^{\langle \vc_l,\vs \rangle}}\\
&=\sum_{n_1 ,\ldots, n_r =0}^{\infty} \frac{\tht_h'' (n_1,\ldots , n_r)}
{\prod_{l=1}^m (\lambda_{l1}(n_1+u_1) +\cdots+\lambda_{lr}(n_r+u_r) )^{\langle \vc_l,\vs \rangle}},
\end{split}
\end{equation*}
where $\tht_h'' (n_1,\ldots , n_r):=\tht (n_1,\ldots , n_r)\theta_h'(n_1,\ldots , n_r)$ satisfying
\begin{equation}
|\tht_h'' (n_1,\ldots , n_r)|=|\tht_h (n_1,\ldots , n_r)\theta'(n_1,\ldots , n_r)|=O\left((n_1+ \cdots +n_r)^{(m+1)\ep}\right).
\label{eq:esth''}
\end{equation}

Hence we have $(\partial / \partial s_h) Z_S (\vs) \in {\mathcal{Z}_S}$ for $1\le h\le m$. 
Inductively, we also have $Z_S^{(\vk)} (\vs) \in {\mathcal{Z}_S}$ for any $\vk\in\Z_{\ge 0}^d$. 
This completes the proof.
\end{proof}

Some important examples of $Z_S(\vs)$ are the following.

\begin{example}
$(i)$ When $d=m=r=\lambda_{11} = u_1= c_1 =1$, $\tht(n)=-\log(n+1)$, we have 
\begin{equation}
Z_S (\vs) = - \sum_{n=1}^{\infty} \frac{\log n}{n^s} =  \zeta'(s), 
\label{eq:zedeir}
\end{equation}
the derivative of the Riemann zeta function which is contained in Theorem \ref{th:de}. \\
\pn
$(ii)$ When $d=m=r=\lambda_{11} = c_1 =1$ and $\tht(n)= e^{2\pi {\rm i} v n}$, where $v \in {\mathbb{R}}$, we have 
\begin{equation}
Z_S (\vs) =  \sum_{n=0}^{\infty} \frac{e^{2\pi {\rm i} v n}}{(n+u)^s}, 
\label{eq:lerchzeta}
\end{equation}
the Lerch zeta function which is a generalization of the Hurwitz zeta function. When $\theta (n) = q^n$, where $q$ is a complex number and $0<|q|<1$, then
\begin{equation}\label{eq:LT}
Z_S (\vs) =  \sum_{n=0}^{\infty} \frac{q^n}{(n+u)^s}, 
\end{equation}
the Lerch transcendent function.\\
\pn
$(iii)$ When $d=m=r$, $\lambda_{11} = \cdots = \lambda_{mr}=1$, $\vc_1 = (1,0,\ldots ,0), \ldots , \vc_m = (0,\ldots, 0,1)$, 
 $\tht (n_1, \ldots , n_m) = 1$, $n_1 > \cdots >n_r >0$, and $\tht (n_1, \ldots , n_m) = 0$, otherwise, then one has
\begin{equation}
\begin{split}
Z_S (\vs)&= \sum_{n_1 > \cdots >n_r >0}^{\infty} \frac{1}{(n_1+u_1)^{s_1}(n_2+u_2)^{s_2}\cdots (n_r+u_r)^{s_r}} \\
&= \sum_{n_1, \ldots , n_r = 1}^{\infty} 
\frac{1}{(n_1+\cdots +n_r+u_1)^{s_1} (n_2+\cdots +n_r+u_2)^{s_2}\cdots (n_r+u_r)^{s_r}} ,
\end{split}
\label{eq:ezhdef}
\end{equation}
the Euler-Zagier-Hurwitz type of multiple zeta function. 
\end{example}

%%%%%%%%%%%%%%%%%%%%%%%%%%%%%%%%%%%%%%%%%%%
\subsection{Shintani zeta distributions on $\rd$}
%%%%%%%%%%%%%%%%%%%%%%%%%%%%%%%%%%%%%%%%%%%

By following the history of zeta distributions on $\R$, we define a new probability distribution on $\rd$ generated by $Z_S$ and consider their infinite divisibility.

Let $\tht (n_1, \ldots , n_r)$ be a nonnegative or nonpositive definite function and again
write $\vc_l=(c_{l1},\dots,c_{ld}) \in {\mathbb{R}}^d$ in Definition \ref{Def}. 

\begin{definition}[Multidimensional Shintani zeta distribution]\label{DefSD}
For $(n_1,\dots , n_r)\in\mathbb{Z}_{\ge 0}^{r}$ and $\vsig$ satisfying 
$\min_{1\le l\le m}\langle \vc_l, \vsig\rangle >r/m$ as in Theorem \ref{th:sc}, 
we define a multidimensional Shintani zeta random variable $X_{\vsig}$ 
with probability distribution on $\rd$ given by
\begin{align*}
{\rm Pr} \Biggl(X_{\vsig}= \biggl(\,& -\sum_{l=1}^m c_{l1} \log \bigl(\lambda_{l1}(n_1+u_1) +\cdots+\lambda_{lr}(n_r+u_r)\bigr),\\
&\dots , -\sum_{l=1}^m c_{ld} \log \bigl(\lambda_{l1}(n_1+u_1) +\cdots+\lambda_{lr}(n_r+u_r)\bigr)\,\biggr) \Biggr)\\
=\ \ \ \ &\frac{\tht (n_1,\ldots , n_r)}{Z_S(\vsig)}
\prod_{l=1}^m \bigl(\lambda_{l1}(n_1+u_1) +\cdots+\lambda_{lr}(n_r+u_r) \bigr)^{-\langle \vc_l,\vsig \rangle}.
\end{align*}
\end{definition}

It is easy to see these distributions are probability distributions since
\begin{align*}
\frac{\tht (n_1,\ldots , n_r)}{Z_S(\vsig)} \prod_{l=1}^m \bigl(\lambda_{l1}(n_1+u_1) 
+\cdots+\lambda_{lr}(n_r+u_r) \bigr)^{-\langle \vc_l,\vsig \rangle} \ge 0
\end{align*}
for each $(n_1,\dots , n_r)\in\mathbb{Z}_{\ge 0}^{r}$ when $\tht (n_1, \ldots , n_r)$ is nonnegative or nonpositive definite, and 
\begin{align*}
\sum_{n_1 ,\ldots, n_r =0}^{\infty} \frac{\tht (n_1,\ldots , n_r)}{Z_S(\vsig)}
\prod_{l=1}^m \bigl(\lambda_{l1}(n_1+u_1) +\cdots+\lambda_{lr}(n_r+u_r) \bigr)^{-\langle \vc_l,\vsig \rangle}
=\frac{Z_S(\vsig)}{Z_S(\vsig)}=1
\end{align*}
by Theorem \ref{th:sc}.

The characteristic function of $X_{\vsig}$ is as follows.

\begin{theorem}\label{th:scf}
Let $X_{\vsig}$ be a multidimensional Shintani zeta random variable.
Then its characteristic function $f_{\vsig}$ is given by
\begin{align*}
f_{\vsig}(\vt)=\frac{Z_S(\vsig +{\rm i}\vt)}{Z_S(\vsig)}, \q \vt\in\rd.
\end{align*}
\end{theorem}

\begin{proof}
By the definition, we have, for any $\vt\in\rd$,
\begin{align*}
f_{\vsig}(\vt)
&=\sum_{n_1 ,\ldots, n_r =0}^{\infty}e^{{\rm i}\langle \vt,X_{\vsig}\rangle}
\frac{\tht (n_1,\ldots , n_r)}{Z_S(\vsig)} \prod_{l=1}^m \bigl(\lambda_{l1}(n_1+u_1) +\cdots+\lambda_{lr}(n_r+u_r) \bigr)^{-\langle \vc_l,\vsig \rangle} \\
&=\frac{Z_S(\vsig +{\rm i}\vt)}{Z_S(\vsig)}.
\end{align*}
\end{proof}

This theorem shows that Definition \ref{DefSD} gives a new generalization of 
zeta distributions on $\R$ mentioned in Section 1.3 to $\rd$-valued.

\vskip3mm

In the following, we give some simple examples of probability distributions on $\R$ in this class.

\begin{example}\label{ex:simp}
Let $m=d=r=1$ and put $\ld_{11}=\ld$, $u_1=u$ and $c_1=c$, where $\ld, u>0$ and $c\in\R$, then Shintani zeta distribution contains the following distributions.
\begin{enumerate}
\item{}
Delta measure.
\item{}
Binomial distribution.
\item{}
Poisson distribution.
\end{enumerate}
\end{example}

These examples can be obtained as follows.

\begin{proof}[Proof of Example \ref{ex:simp}]
First we show $(i)$.
Let $\sigma >c^{-1}$, $\theta(0)=\theta\in\R$ and $\theta(n)=0$, $n \ge 1$.
Then we have, for any $t\in\R$,
$$
\frac{Z_S(\sigma +{\rm i}t)}{Z_S(\sigma)} = \frac{\theta (\ld u)^{-c(\sigma +{\rm i}t)}}{\theta (\ld u)^{-c\sigma}}
 = (\ld u)^{-{\rm i}ct} = e^{{\rm{i}}(-c\log (\ld u))t}.
$$
Hence we obtain a characteristic function of a delta measure at $-c\log (\ld u)$.

Next we show $(ii)$.
Let $j\in\N\setminus\{1\}$, $\ld =u=1$, $c=-\left(\log j\right)^{-1}$, $\sigma <-\log j$ and $\phi(j) >0$.
Put
$$
p = \frac{\phi(j) j^{\sigma /\log j}}{1 +\phi(j)j^{\sigma /\log j}}, \qquad q =1-p = \frac{1}{1+\phi(j) j^{\sigma /\log j}}.
$$
Then we have, for any $t\in\R$,
\begin{align}
\frac{1+\phi(j) j^{-c(\sigma+{\rm{i}}t)}}{1+\phi(j) j^{-c\sigma}}
= \frac{1+\phi(j) j^{-c\sigma}e^{-{\rm i}tc\log j}}{1+\phi(j) j^{-c\sigma}}=\frac{1+\phi(j) j^{\sigma /\log j}e^{{\rm i}t}}{1+\phi(j) j^{\sigma /\log j}}= pe^{{\rm{i}}t}+q.\label{eq:pq}
\end{align}

Let $K \in {\mathbb{N}}$, $\theta(j^k-1)= {}_KC_k \left(\phi(j)\right)^k$, $k\in \{0,1,\ldots, K\}$, and $\theta (n) =0$, otherwise.
Now consider a Shintani zeta function of the form
$$
Z_S(s)= \sum_{k=0}^{K}\frac{{}_KC_k \left(\phi(j)\right)^k}{ \left(j^k\right)^{cs}}=(1+\phi(j) j^{-cs})^K, \q s\in\mathbb{C}.
$$
By \eqref{eq:pq}, we have, for any $t\in\R$,
$$
\frac{Z_S(\sigma +{\rm i}t)}{Z_S(\sigma)}= \frac{\left(1+\phi(j) j^{-c(\sigma+{\rm{i}}t)}\right)^K}{\left(1+\phi(j) j^{-c\sigma}\right)^K}=(pe^{{\rm{i}}t}+q)^K.
$$
Hence we obtain a characteristic function of a binomial distribution with parameter $K$.

Finally, we show $(iii)$.
Let $a\in\R$, $j\in\N\setminus \{1\}$, $\ld =u=1$, $c=-(\log j)^{-1}$ and $\sigma <-\log j$.
Put $\theta(0)=1$, $\theta(j^k-1)= j^{ak}/k!$ and $\theta (n) =0$, otherwise.
Then we have, for any $t\in\R$,
$$
Z_S(\sigma +{\rm i}t) = \sum_{k=0}^\infty \frac{\left(j^{ak}\right)\left(j^k\right)^{-c(\sigma+{\rm{i}}t)}}{k!} = 
\sum_{k=0}^\infty \frac{\left(j^{a-c\sigma }\right)^k \left(e^{-{\rm{i}}ct\log j}\right)^k}{k!} = \exp \left( j^{a+\sigma /\log j} e^{{\rm{i}}t}\right) .
$$
Therefore one has
$$
\frac{Z_S(\sigma +{\rm i}t)}{Z_S(\sigma)} = 
\frac{\exp \left( j^{a+\sigma /\log j} e^{{\rm{i}}t}\right)}{\exp \left( j^{a+\sigma /\log j} \right)} =
\exp \left( j^{a+\sigma /\log j} (e^{{\rm{i}}t}-1) \right), \q t\in\R.
$$
Hence we obtain a characteristic function of a Poisson distribution with mean $j^{a+\sigma /\log j}>0$.
\end{proof}

\begin{remark}
Note that we have added some conditions for $\sigma$ in Example \ref{ex:simp} as to adjust to the definition of the Shintani zeta distribution.
However, these examples can be given under the condition with any $\sigma\in\R$ where the corresponding Shintani zeta functions convergence absolutely as well.
\end{remark}

\vskip3mm

For $\vk\in\Z_{\ge 0}^d$, put
$$
f_{\vsig}^{(\vk)}(\vt):=\frac{Z_S^{(\vk)}(\vsig +{\rm i}\vt)}{Z_S^{(\vk)}(\vsig)}, \q \vt\in\rd.
$$
Note that $Z_S^{(\vk)}\in\mathcal{Z}_S$ by Theorem \ref{th:de}.
Then, we have the following. 

\begin{theorem}
If $f_{\vsig}^{(\vz)}(\vt)$ is a characteristic function of a multidimensional Shintani zeta random variable $X_{\vsig}$, 
$\sum_{j=1}^r \lambda_{lr} u_r \ge 1$ and $c_{l1}, \ldots , c_{ld}$ have the same sign for each $1 \le l \le m$,
then $f_{\vsig}^{(\vk)}$ is also a characteristic function for any $\vk\in\Z_{\ge 0}^d$.
\end{theorem}

\begin{proof}
Let $X_{\vsig}$ be a multidimensional Shintani zeta random variable.
From (\ref{eq:th'}) in the proof of Theorem \ref{th:de}, the functions $\theta_h'(n_1,\ldots , n_r)$ and 
$\tht'' (n_1, \ldots , n_r) := \tht (n_1, \ldots , n_r)$ $\theta_h'(n_1,\ldots , n_r)$ are nonnegative or nonpositive definite when $\sum_{j=1}^r \lambda_{lr} u_r \ge 1$ and $c_{l1}, \ldots , c_{ld}$ have the same sign. 
Note that $\theta_h'(n_1,\ldots , n_r)$ have the opposite sign of 
$c_{l1}, \ldots , c_{ld}$ when $\sum_{j=1}^r \lambda_{lr} u_r \ge 1$.
Moreover, we have \eqref{eq:esth''}. 
Thus, $X''_{\vsig}$, replaced $\tht (n_1, \ldots , n_r)$ of $X_{\vsig}$ by $\tht'' (n_1, \ldots , n_r)$, is also a multidimensional Shintani zeta random variable.
Hence, by Theorem \ref{th:scf}, $f_{\vsig}^{(\vk)}$ is a characteristic function when $||\vk||=1$. 
Inductively, we also have the case $||\vk||>1$. 
\end{proof}

The following is well-known.

\begin{proposition}[See, e.g.\,\cite{S99}]\label{pro:Sato}
Let $\mu$ be a probability measure on $\rd$.\\
$(i)$ Let $n$ be a positive even integer. 
If $\wh\mu(\vt)$ is of class $C^n$ in a neighborhood of the origin, then $\mu$ has finite absolute moment of order $n$.\\ 
$(ii)$ If $\mu\in I(\rd)$, then $\wh\mu$ does not have zeros that is $\wh\mu (\vt)\neq 0$ for any $\vt\in\rd$.
\end{proposition}

Next we give the moment condition of the multidimensional Shintani zeta distribution.

\begin{theorem}\label{th:mo}
Let $k\in\N$ and $X_{\vsig}$ be a multidimensional Shintani zeta random variable, then we have
$$
E|X_{\vsig}|^{2k} <\infty.
$$
\end{theorem}
\begin{proof}
Let $f_{\vsig}$ be the characteristic function of $X_{\vsig}$.
By following the proof of Theorem \ref{th:de}, we have, for any $n\in\N$ and $\vt\in\rd$, $f_{\vsig}$ is of class $C^n$.
Thus by Proposition \ref{pro:Sato} $(i)$, for any $n=2k$, we obtain $E|X_{\vsig}|^{2k} <\infty.$
\end{proof}

We also have the following by Proposition \ref{pro:Sato} $(ii)$. 

\begin{theorem}\label{th:notinf}
Multidimensional Shintani zeta distributions with $f_{\vsig}$ having zeros in the region $\min_{1\le l\le m}\Re\langle \vc_l, \vs\rangle >r/m$ are not infinitely divisible. 
\end{theorem}

\begin{remark}
A similar distribution is also defined in the same way by the former definition of the multidimensional Shintani zeta function in \cite{AN11k}.
The region of $\vsig$ is expanded to $\min_{1\le l\le m}\langle \vc_l,\vsig \rangle >r/m$ from $-\ep +\sum_{l=1}^m\langle \vc_l,\vsig \rangle >r$ for some $\ep >0$.
Also Theorems \ref{th:scf} and \ref{th:notinf} are shown when $\vsig$ satisfies $-\ep +\sum_{l=1}^m\langle \vc_l,\vsig \rangle >r$ in \cite{AN11k}.
Since Theorem \ref{th:de} is a new result, Theorem \ref{th:mo} is completely new one.
\end{remark}

By Theorem \ref{th:notinf}, zeta distributions generated by following functions when they have zeros for some $\vsig$ are not infinitely divisible: 
\begin{enumerate}
\item Partial zeta functions $\sum_{n \le N} n^{-s}$ for some suitable integer $N$,
\item The derivative of the Riemann zeta function (\ref{eq:zedeir}), 
\item The Lerch transcendent function \eqref{eq:LT} proved in \cite{GG}.
\item Some Dirichlet series with periodic coefficients, which contains the Hurwitz zeta functions 
with $u \ne 1/2$ and $u$ are rational, treated by Saias and Weingartner \cite{SW}, 
\item Euler-Zagier-Hurwitz type of multiple zeta functions (\ref{eq:ezhdef}) 
when $u_1,\ldots ,u_r$ are algebraically independent over rationals proved in \cite[Proposition 3.2]{Nakamura4}. 
\end{enumerate}

%%%%%%%%%%%%%%%%%%%%%%%%%%%%%%%
\section{Relation with multidimensional polynomial Euler products}
%%%%%%%%%%%%%%%%%%%%%%%%%%%%%%%

In \cite{AN12e}, we have introduced multidimensional polynomial Euler products as to expand the Riemann zeta distribution to $\rd$-valued with infinite divisibility.
Its definition is as follows.

\begin{definition}[Multidimensional polynomial Euler product, $Z_E(\vs)$ (\cite{AN12e})]\label{def:EP}
Let $d,m\in\N$ and $\vs\in\mathbb{C}^d$.
For $-1 \le \alpha_l(p) \le 1$ and $\va_l \in {\mathbb{R}}^d$, $1\le l\le m$ and $p\in\Prime$,
we define multidimensional polynomial Euler product given by
\begin{equation}
Z_E (\vs) = \prod_p \prod_{l=1}^m \left( 1 - \alpha_l(p) p^{-\langle \va_l,\vs\rangle} \right)^{-1}.
\label{eq:def1}
\end{equation}
\end{definition}

Let ${\mathcal{Z}}_E$ be the set of functions of $Z_E$. 
Then we have the following.

\begin{theorem}\label{th:shieu1}
It holds that
$$
{\mathcal{Z}}_E \subset {\mathcal{Z}}_S.
$$ 
\end{theorem}

For the proof of this theorem, we use the following lemma.

\begin{lemma}[{\cite[Lemma 2.2]{Steuding1}}]\label{lm:A}
Suppose that a function $L(s)$ is given by 
$$
L(s)= \sum_{n=1}^\infty \frac{A(n)}{n^s} = \prod_p \prod_{l=1}^m \biggl( 1-\frac{\alpha_l(p)}{p^s} \biggr)^{-1} .
$$
Then $A(n)$ is multiplicative and
$$
A(n) = \prod_{p|n} \sum_{\substack{0 \le k_1, \ldots , k_m \\ k_1+ \cdots +k_m = \nu(n;p)}}
\prod_{l=1}^m \alpha_l(p)^{k_l},
$$
where $\nu(n;p)$ is the exponent of the prime $p$ in the prime factorization of the integer $n$. 
Moreover, if $|\alpha_l(p)| \le 1$ for $1 \le l \le m$ and all primes $p$, then $|A(n)| = O(n^{\varepsilon})$ for any $\varepsilon >0$, and vice versa.
\label{lem:st2.2}
\end{lemma}

\begin{proof}[Proof of Theorem \ref{th:shieu1}]
We have, for any $\vs\in\mathbb{C}^d$ satisfying $\min_{1\le l\le m}\Re \langle \va_l,\vs\rangle >1$ and $1\le l\le m$,
$$
\prod_{p} \bigl( 1 - \alpha_{l}(p) p^{-\langle \va_l,\vs\rangle} \bigr)^{-1} =
\prod_{p} \biggl( 1 + \sum_{k=1}^\infty \frac{\alpha_{l}(p)^k}{p^{k\langle \va_l,\vs\rangle}} \biggr) =
\sum_{n=1}^{\infty} \frac{A_l(n)}{n^{\langle \va_l,\vs\rangle}}, 
$$
where
$$
A_l (n)= \prod_{p|n} \alpha_{l}(p)^{\nu(n;p)}.
$$
This equality implies Lemma \ref{lm:A} with $m=1$.
Note that $|A_l(n)| \le 1$ since $-1 \le \alpha_{l}(p) \le 1$. 
Thus we have, for any $\vs\in\mathbb{C}^d$ with $\min_{1\le l\le m}\Re \langle \va_l,\vs\rangle >1$,
$$
\prod_p \prod_{l=1}^m \bigl( 1 - \alpha_{l}(p) p^{-\langle \va_l,\vs\rangle} \bigr)^{-1} =
\prod_{l=1}^m \sum_{n_l=1}^{\infty} \frac{A_l(n_l)}{n_l^{\langle \va_l,\vs\rangle}} =
\sum_{n_1, \ldots ,n_m=1}^\infty
\frac{A_1(n_1)}{n_1^{\langle \va_1,\vs\rangle}} \cdots \frac{A_m(n_m)}{n_m^{\langle \va_m,\vs\rangle}}.
$$
Obviously, we have $\prod_{l=1}^m |A_l(n_l)| \le 1$ and $\left|\prod_{l=1}^m |A_l(n_l)|\right|= O((n_1+ \cdots +n_r)^{\ep})$ for any $\ep >0$.
Therefore 
$$
\mathcal{Z}_E\ni\prod_p \prod_{l=1}^m \bigl( 1 - \alpha_{l}(p) p^{-\langle \va_l,\vs\rangle} \bigr)^{-1}=
\sum_{n_1, \ldots ,n_m=1}^\infty \frac{A_1(n_1)}{n_1^{\langle \va_1,\vs\rangle}} \cdots \frac{A_m(n_m)}{n_m^{\langle \va_m,\vs\rangle}}\in \mathcal{Z}_S.
$$
\end{proof}

In the next, we consider this relation by treating some simple zeta functions.

\begin{example}
Let $d_k (n)$, $k=2,3,4,\ldots ,$ denote the number of ways of expressing $n$ as a product of $k$ factors, expression with the same factors in a different order being counted as different. 
Then we have
$$
\prod_p \bigl( 1 - p^{-s} \bigr)^{-k} = \zeta^k (s) = 
\sum_{m_1=1}^\infty \frac{1}{m_1^s} \cdots \sum_{m_k=1}^\infty \frac{1}{m_k^s} =
\sum_{n=1}^\infty \frac{1}{n^s} \sum_{m_1 \cdots m_k=n} 1 = \sum_{n=1}^\infty \frac{d_k(n)}{n^s}
$$
where $\Re (s)>1$ (see, for example \cite[(1.2.2)]{Tit}). 
\end{example}

By applying Lemma \ref{lem:st2.2}, we can give another example of Shintani zeta distribution on $\R$ related to number theory.
As to give it, we use the following well-known function.

\begin{definition}[Dedekind zeta function of $\mathbb{Q}({\rm i})$ (see, e.g.\,\cite{Cohen})]
Let $\mathbb{Q}({\rm i})$ be a quadratic field of discriminant $-1$.
The Dedekind zeta function of $\mathbb{Q}({\rm i})$ is a function of a complex variables $s=\sigma +{\rm i}t$, for $\sigma >1$ given by
$$
\zeta_{{\mathbb{Q}}({\rm i})} (s) := \zeta (s) L(s), 
$$
where
\begin{equation}\label{f:L}
L(s) := \sum_{n=1} \frac{\chi_{-4}(n)}{n^s}, \qquad \chi_{-4}(n) :=
\begin{cases}
1 & n \equiv 1 \mod 4,\\
-1 &  n \equiv 3 \mod 4,\\
0  & n \equiv 0,2 \mod 4.
\end{cases}
\end{equation}
\end{definition}

\vskip3mm

Now we have the following.

\begin{example}\label{ex:5}
The Dedekind zeta function generates a characteristic function which belongs to the multidimensional Shintani zeta distribution. 
\end{example}

\begin{proof}

This function is closely related to number theory. 
By the definition, we have
$$
\zeta_{{\mathbb{Q}}({\rm i})} (s) = \sum_{m,n=1} \frac{\chi_{-4}(n)}{m^sn^s}.
$$
Thus we can not see whether this function generates a distribution in the sense of Shintani or not in this way. 
However, it is known that (see, for example \cite[p.~221]{Cohen})
$$
\zeta_{{\mathbb{Q}}({\rm i})} (s) = \frac{1}{4} \sum_{(m,n) \in {\mathbb{Z}}^2 \setminus (0,0)} \frac{1}{(m^2+n^2)^s} = 
\sum_{n=1} \frac{A(n)}{n^s},
$$
where $A(n)$ is nonnegative definite coefficient given by
$$
A(n) := \frac{1}{4} \# \{ (m_1,m_2) \in {\mathbb{Z}}^2 : m_1^2+m_2^2 =n\} = \sum_{d \mid n} \chi_{-4}(d) .
$$
Moreover, we obtain $A(n) = O(n^{\varepsilon})$ by Lemma \ref{lem:st2.2}. 
Therefore now we can see that $\zeta_{{\mathbb{Q}}({\rm i})} (s)$ generates a distribution in the sense of Shintani. 
\end{proof}

\begin{remark}
We have also shown that the Dedekind zeta function also generates a multidimensional compound Poisson characteristic function in view of the multidimensional polynomial Euler products.
(See, \cite[Example 4.2 (i)]{AN12e}.)
It should be noted that this function is one of the rare case we can show that the normalized function belongs to both multidimensional Shintani zeta distribution and compound Poisson distribution generated by the multidimensional polynomial Euler products since it is difficult to obtain positive $A(n)$ for general zeta functions. 
\end{remark}

Throughout this section, we have considered the relation between series representations and Euler products of multivariable zeta functions.
We also have noted that Shintani zeta distributions contain binomial and Piosson distributions as in Example \ref{ex:simp}.
Though, it is still difficult to treat zeta distributions on $\rd$ only by series representations.
As mentioned, we have introduced infinitely divisible zeta distributions on $\rd$ by Euler products in \cite{AN12e}.
However, they also include products which generate not infinitely divisible $\rd$-valued characteristic functions and not even to generate 
characteristic functions.
To obtain more detail of behaviors of multivariable zeta functions in this view, 
we have studied them by treating multivariable finite Euler products as a simple case in \cite{AN12q}.

%%%%%%%%%%%%%%%%%%%%%%%%%%
 
\end{document}